\documentclass[a4paper,12pt]{amsart}

\usepackage[latin1]{inputenc}
\usepackage[T1]{fontenc}
\usepackage[english]{babel}

\usepackage{mathrsfs}
\usepackage{amscd}
\usepackage{amsfonts}
\usepackage{amsmath}
\usepackage{amssymb}
\usepackage{amstext}
\usepackage{amsthm}
\usepackage{amsbsy}

\usepackage{xspace}
\usepackage[all]{xy}
\usepackage{graphicx}
\usepackage{url}
\usepackage{latexsym}

\makeatletter
\newcommand*{\rom}[1]{\expandafter\@slowromancap\romannumeral {\sharp}1@}
\makeatother

\theoremstyle{definition}

\newtheorem{fact}{fact}

\newtheorem{thm}[fact]{Theorem}
\newtheorem{lemma}[fact]{Lemma}
\newtheorem{prop}[fact]{Proposition}
\newtheorem{corollary}[fact]{Corollary}
\newtheorem{defini}[fact]{Definition}
\newtheorem{remark}[fact]{Remark}

\newtheorem{question}[fact]{Question}

\title{Space and Time Complexity for Infinite Time Turing Machines}
\author{Merlin Carl}

\begin{document}
	
	\begin{abstract}
		We consider notions of space complexity for Infinite Time Turing Machines (ITTMs) that were introduced by B. L\"owe in \cite{Loe} and studied further by Winter in \cite{Wi1} and \cite{Wi2}.
		We answer several open questions about these notions, among them whether low space complexity implies low time complexity (it does not) and whether one of the equalities P=PSPACE, P$_{+}=$PSPACE$_{+}$ and P$_{++}=$PSPACE$_{++}$ holds for ITTMs (all three are false). We also show various separation results between space complexity classes for ITTMs. This considerably expands our earlier observations on the topic in section 7.2.2 of \cite{Ca2}, which appear here as Lemma $6$ up to Corollary $9$.
	\end{abstract}
	
\maketitle

\section{Introduction}

Complexity theory for ITTMs was started by Schindler in \cite{Schi} with the observation that, with natural analogues of the classes P and NP for ITTMs, we have P$\neq$NP for ITTMs.
While time complexity for ITTMs received some further attention, e.g., in \cite{DHS}, a satisfying analysis of space complexity was hindered by the fact that all but the most trivial ITTM-computations use the whole tape length $\omega$ and thus have the same space usage.

In \cite{Loe}, L\"owe suggested an alternative view: Given that a tape of length $\omega$ can, via coding, simulate tapes of much larger order types, the complexity of the tape contents during the computations rather than the number of used cells should be used as a measure of space usage. This idea was further pursued by Winter in \cite{Wi1} and \cite{Wi2}. 

One general conjecture in \cite{Loe} is that sets decidable by computations that consist entirely of `simple' snapshots can also be decided in short time, which would mean that 
time and space complexity for ITTMs are strongly related. 

In this paper, we we will disprove this conjecture by showing that, if $\sigma$ denotes the minimal ordinal such that $L_{\sigma}$ is $\Sigma_{1}$-elementary in $L$, then, for any $\alpha<\sigma$, there are sets of real numbers that are ITTM-decidable with extremely simple snapshots, but not with time bounded by $\alpha$. As we will also see that uniform time bounds  on ITTM-decision times are always $<\sigma$, there seems to be no influence of content complexity on time complexity for ITTMs. As a byproduct, we obtain that none of the equalities P=PSPACE, P$_{+}=$PSPACE$_{+}$ and P$_{++}=$PSPACE$_{++}$ holds for ITTMs (these classes will be defined below); this answers an open question by B. L\"owe (see \cite{BIWOC}, p. 42) and another one by J. Winter (ebd).

Moreover, we will show that many of the space complexity classes defined in \cite{Loe} and \cite{Wi1} are distinct.

\section{Preliminaries}


For Infinite Time Turing Machines (ITTMs) and basic notions like writability (of a real number) or decidability (of a set of real numbers), we refer to \cite{HL}, \cite{W1}, \cite{W2}. We will also freely use notions like writability, clockability, eventual writability, accidental writability, decidability and recognizability, all of which can be found in \cite{HL}, along with the following well-known facts about ITTMs:

\begin{itemize}
	\item (Welch, see \cite{W1}) For each $x\subseteq\omega$, there are ordinals $\lambda^{x}<\zeta^{x}<\Sigma^{x}$ such that a real number $y\subseteq\omega$ is ITTM-writable/eventually writable/accidentally writable in the oracle $x$ if and only if $y\in L_{\lambda^{x}}[x]/L_{\zeta^{x}}[x]/L_{\Sigma^{x}}[x]$; moreover, $(\lambda^{x},\zeta^{x},\Sigma^{x})$ is lexically minimal with the property that $L_{\lambda^{x}}[x]\prec_{\Sigma_{1}}L_{\zeta^{x}}[x]\prec_{\Sigma_{2}}L_{\Sigma{x}}[x]$.
	\item (Welch, \cite{W1}) The supremum of the ITTM-writable ordinals coincides with the supremum of the ITTM-clockable ordinals.
	\item (Hamkins and Lewis, \cite{HL}, Welch, \cite{W1}) For any ITTM-program $P$ and any $x\subseteq\omega$, $P^{x}$ will either halt in $<\lambda^{x}$ many steps or run into a strong loop before time $\Sigma^{x}$. Here, a strong loop means that a snapshot $s$ is repeated at limit times and that, between the two occurences of $s$, all occuring snapshots weakly majorized $s$ in all components. In particular, if a cell contains $0$ in $s$, then it contained $0$ for all snapshots occuring between the two occurences of $s$.
	\item There are no ITTM-recognizable real numbers in $L_{\Sigma}\setminus L_{\lambda}$.
	\item The minimal $L$-level that contains all ITTM-recognizable real numbers is $L_{\sigma}$, where $\sigma$ is minimal with the property $L_{\sigma}\prec_{\Sigma_{1}}L$.
	\item (Folklore) If $c\subseteq\omega$ is ITTM-recognizable, then $c\in L_{\lambda^{c}}$.
\end{itemize}

A model of transfinite computability that is less well known than ITTMs are Infinite Time Register Machines (ITRMs), introduced by Koepke \cite{ITRM1}. These will play an important role in several of the proofs below. 
For the definition and main results on Infinite Time Register Machines (ITRMs), we refer to \cite{ITRM1}, \cite{ITRM2}. 

\begin{lemma}{\label{ITTM rec multitape simulation}}
Let $P$ be an ITTM-program using a finite number $n$ of scratch tapes and let $x\subseteq\omega$ be such that $P^{x}$ uses only recursive snapshots. Then $P$ can be simulated by an ITTM-program $Q$ with a single scratch tape that uses only recursive scratch tapes.
\end{lemma}
\begin{proof}
Split the tape into $n$ portions and write the $i$th bit of the $j$th tape to the cell with index $in+j$. 
\end{proof}

If $\alpha$ is an ordinal, an $\alpha$-ITRM works like an ITRM, but the bound on the register contents is $\alpha$ rather than $\omega$. They were suggested by Koepke in \cite{ordinal computability}, but have received comparably little attention so far (see, however, \cite{alpha itrms} and \cite{Ca2} for some recent progress on determining their computational strength).

\begin{lemma}{\label{alpha ITRM ITTM simulation}}
Let $\alpha<\omega_{1}^{\text{CK}}$, and let $P$ be an $\alpha$-ITRM-program. Then there is an ITTM-program $Q$ such that, for any $x\subseteq\omega$, $Q^{x}$ has the same halting behaviour and output as $P^{x}$ and only produces recursive snapshots.
\end{lemma}
\begin{proof}
We prove this in general, although only the somewhat simpler case $\alpha=\omega$ will be needed below.

$Q$ will simply simulate $P$ on the input $x$. To this end, we let $Q$ use one tape for each register used by $P$ and then use Lemma \ref{ITTM rec multitape simulation}. Fix a recursive bijection $f:\omega\rightarrow\alpha$. Now, the $i$th cell of a scratch tape will represent $f(i)$. The register content $\beta<\alpha$ in register $k$ will be represented by writing $1$ to all cells of the $k$th scratch tape that have an index $i$ with $f(i)<\beta$ and $0$ to all other cells. Clearly, this will be a recursive real number.

On a further scratch tape, the active program line is stored by just writing $1$s to the first $k$ cells when $k$ is the active program line and $0$ to all other cells.

We show how to simulate each of the register machine commands:

\begin{itemize}
\item whether a register contains $0$ can be seen by checking the cell corresponding to $0$

\item a conditional jump is simulated by writing the new active program line to the corresponding scratch tape, from left to right.

\item the COPY instruction from register $i$ to register $j$ is simulated by copying the content of scratch tape $i$ to scratch tape $j$ bit by bit, from left to right. Clearly, if the contents of both tapes were recursive to begin with, they will remain so.

\item the incrementation operation is carried out on the $k$th register by searching through $\omega$ for the minimal $\iota<\alpha$ such that the $f^{-1}(\iota)$th cell of the $k$th tape contains $0$ and replacing its content with $1$. Clearly, this can be done with only recursive snapshots, as $f$ is recursive. 

\item whether or not the $k$th tape only contains $1$s can be easily tested with a flag routine. if that is the case, replace them by $0$ one by one, starting from the left. again, this will only use recursive snapshots.

\end{itemize}

Note that the simulation will be automatically correct at limits.

\end{proof}


We recall the Jensen-Karp-theorem from Jensen and Karp, \cite{JK}:

\begin{thm}{\label{jk}}
Let $\phi$ be a $\Sigma_{1}$-formula (possibly using real parameters) and let $\alpha$ be a limit of admissible ordinals such that $V_{\alpha}\models\phi$. Then $L_{\alpha}\models\phi$.
\end{thm}

We will write WO for the set of real numbers that encode well-orderings. As usual, we will denote the next admissible ordinal after an ordinal $\alpha$ by $\alpha^{+}$ and the next limit of admissible ordinals after $\alpha$ by $\alpha^{+\omega}$. 
A basic result about ITRMs is that WO is ITRM-decidable, see Koepke and Miller \cite{ITRM2}.

Moreover, we recall the following statement from \cite{Ca1}:

\begin{thm}{\label{recogdistribution}}
Every ITRM-recognizable real number is contained in $L_{\sigma}$. Moreover, for any $\alpha<\sigma$, there is an ITRM-recognizable real number in $L_{\sigma}\setminus L_{\alpha}$.
\end{thm}

The basic notions of space complexity theory for ITTMs were defined by B. L\"owe in \cite{Loe}, while those of time complexity for ITTMs were given by Schindler in \cite{Schi}. We give a brief summary.

\begin{defini}
	Let $X\subseteq\mathfrak{P}(\omega)$.
	\begin{itemize}
	\item For $f:\mathfrak{P}(\omega)\rightarrow\text{On}$, we say that $X$ belongs to TIME$_{f}^{\text{ITTM}}$ if and only if there is an ITTM-program $P$ that decides $X$ and halts in $<f(x)$ many steps on input $x$. If $f$ is constant with value $\alpha$, we will write TIME$_{\alpha}^{\text{ITTM}}$. P$^{\text{ITTM}}$ then denotes TIME$_{\omega^{\omega}}^{\text{ITTM}}$. If $f$ is the function $x\mapsto\omega_{1}^{\text{CK},x}$, we write P$^{\text{ITTM}}_{+}$ for TIME$_{\alpha}^{\text{ITTM}}$. If $f$ is the function $x\mapsto\omega_{1}^{\text{CK},x}+\omega+1$, we write $P^{\text{ITTM}}_{++}$ for TIME$_{\alpha}^{\text{ITTM}}$.
	\item For $f:\mathfrak{P}(\omega)\rightarrow\text{On}$, we say that $X$ belongs to SPACE$_{f}^{\text{ITTM}}$ if and only if there is an ITTM-program $P$ that decides $X$ and such that all snapshots occuring during the computation of $P^{x}$ for any $x\subseteq\omega$ are contained in $L_{f(x)}[x]$. If $f$ is constant with value $\alpha$, we will write SPACE$_{\alpha}^{\text{ITTM}}$. If $P$ is a program witnessing this, we also say, by a slight abuse of notation, that $P$ decides $X$ `with snapshots in $L_{\alpha}$'.\footnote{Note that this means that deciding $X$ `with snapshots in $L_{\alpha}$ allows the snapshots to come from $L_{\alpha}[x]$ for all inputs $x\subseteq\omega$.} PSPACE$^{\text{ITTM}}$ then denotes SPACE$_{\omega^{\omega}}^{\text{ITTM}}$. If $f$ is the function $x\mapsto\omega_{1}^{\text{CK},x}$, we write PSPACE$^{\text{ITTM}}_{+}$ for SPACE$_{\alpha}^{\text{ITTM}}$. If $f$ is the function $x\mapsto\omega_{1}^{\text{CK},x}+\omega+1$, we write PSPACE$^{\text{ITTM}}_{++}$ for SPACE$_{\alpha}^{\text{ITTM}}$. If $X$ can be decided by an ITTM-program $P$ that only uses recursive snapshots on all inputs, we say that $X$ is SPACE$^{\text{ITTM}}_{\text{REC}}$ or that $X$ is ITTM-decidable with recursive snapshots.
	\end{itemize}
\end{defini}

Throughout the paper, we fix a natural enumeration $(P_{i}:i\in\omega)$ of the ITTM-programs.

\section{The connection between space and time complexity for ITTMs}

We now consider the question whether the fact that a set $X\subseteq\mathfrak{P}(\omega)$ can be decided by an ITTM using only `simple' snapshots implies that $X$ is also quickly decidable. Up to Corollary \ref{ITTM time space complexity}, the following is contained in chapter $7$ of the forthcoming monograph \cite{Ca2}. We include proofs for the sake of being self-contained.







We will now show that, for any $\alpha<\sigma$, there are sets $X\subseteq\mathfrak{P}(\omega)$, such that $X$ is ITTM-decidable with recursive snapshots, but does not belong to TIME$^{\text{ITTM}}_{\alpha}$. For $\alpha=\omega^{\omega}$, this answers a question in \cite{Loe} in the negative, namely whether SPACE$^{\text{ITTM}}_{\text{REC}}\subseteq$TIME$^{\text{ITTM}}_{\omega^{\omega}}$.

\begin{lemma}{\label{ITRM recog ITTM simple}} [Cf. \cite{Ca2}, Lemma 7.2.19]
Let $\alpha<\omega_{1}^{\text{Ck}}$.
If $X\subseteq\mathfrak{P}(\omega)$ is $\alpha$-ITRM-decidable, then $X$ is ITTM-decidable with recursive snapshots. 
In particular, if a real number $x$ is ITRM-recognizable, then $\{x\}$ is ITTM-decidable with recursive snapshots. 
\end{lemma}
\begin{proof}
The first claim is an easy consequence of Lemma \ref{alpha ITRM ITTM simulation}, and the second claim is an easy consequence of the first.



\end{proof}

\noindent
\textbf{Remark}: Note that recursiveness is really an understatement when measuring the complexity of the snapshots occuring during the procedure just described. One can hardly imagine anything simpler that goes beyond only allowing finitely many $1$s on the tape.\footnote{Only allowing finitely many $1$s on the tape leads to a weak version of ITTMs studied in \cite{Matzner}; there, it is shown that the subsets of $\omega$ computable by such a machine is $L_{\omega_{1}^{\text{CK}}}\cap\mathfrak{P}(\omega)$.}

\begin{lemma}{\label{index and halting time}} [Cf. \cite{Ca2}, Lemma 7.2.20]

Let $x$ be  a real number such that $x\notin L_{\alpha^{+\omega}}$. Then $x\notin$TIME$^{\text{ITTM}}_{\alpha}$.
\end{lemma}
\begin{proof}
By contraposition. Suppose that $\{x\}\in$TIME$^{\text{ITTM}}_{\alpha}$. Let $P$ be an ITTM-program that recognizes $x$ with uniform time bound $\alpha$. In particular, the computation of $P$ in the oracle $x$ halts in $<\alpha$ many steps. The statement that there is a halting computation of $P$ in some oracle with output $1$ is $\Sigma_{1}$ and holds in $V_{\alpha+\omega}$ and hence in $V_{\alpha^{+\omega}}$. By the Jensen-Karp theorem \ref{jk}, it follows that $x\in L_{\alpha^{+\omega}}$, a contradiction.


\end{proof}

\begin{thm}{\label{space and time for ITTMs}} [Cf. \cite{Ca2}, Theorem 7.2.21]
Let $\alpha<\sigma$. Then TIME$^{\text{ITTM}}_{\alpha}\nsupseteq$SPACE$^{\text{ITTM}}_{\text{REC}}$.
\end{thm}
\begin{proof}
Since $\alpha<\sigma$, we have $\alpha^{+\omega}<\sigma$. Thus, Theorem \ref{recogdistribution} implies that there is an ITRM-recognizable real number $x\in L_{\sigma}\setminus L_{\alpha^{+\omega}}$. Then $\{x\}\in$SPACE$^{\text{ITTM}}_{\text{REC}}\setminus$TIME$^{\text{ITTM}}_{\alpha}$ by Lemma \ref{index and halting time} and Lemma \ref{ITRM recog ITTM simple}.
\end{proof}

Since $L_{\omega+1}$ contains all recursive real numbers, we have:

\begin{corollary}{\label{ITTM time space complexity}} [Cf. \cite{Ca2}, Corollary 7.2.22] 
Let $\alpha<\sigma$. Then TIME$^{\text{ITTM}}_{\alpha}\nsupseteq$SPACE$^{\text{ITTM}}_{\omega+1}$.
\end{corollary}

\subsection{Upper Bounds for Time and Space Complexity}


The above result leaves open the possibility that some version of "low space complexity implies low time complexity" holds for sets that can be decided by ITTMs with a uniform time bound $\geq\sigma$. Clearly, any such nontrivial bound will be $<\omega_{1}$. If we knew that no such bound exists between $\sigma$ and $\omega_{1}$, the question could be regarded as completely settled in the negative. We will now show that the result above is optimal in the sense that all meaningful instances of TIME$^{\text{ITTM}}_{\alpha}$ and SPACE$^{\text{ITTM}}_{\alpha}$ have $\alpha<\sigma$. Thus, any time bound that can occur at all can occur as the minimal decision time bound for a set that is ITTM-decidable with recursive (and much simpler) snapshots. Consequently, there seems to be no connection of the desired kind between time and space complexity for ITTMs.



\begin{defini}
	For an ordinal $\alpha$, let us say that TIME$^{\text{ITTM}}_{\alpha}$ or  SPACE$^{\text{ITTM}}_{\alpha}$ is "inhabited" if and only if
	TIME$^{\text{ITTM}}_{\alpha}\setminus\bigcup_{\iota<\alpha}$TIME$^{\text{ITTM}}_{\iota}\neq\emptyset$
	or
	 SPACE$^{\text{ITTM}}_{\alpha}\setminus\bigcup_{\iota<\alpha}$SPACE$^{\text{ITTM}}_{\iota}\neq\emptyset$, respectively.
\end{defini}


\begin{thm}{\label{sigma bound on inhabited classes}}
	\begin{enumerate}
		\item If TIME$^{\text{ITTM}}_{\alpha}$ is inhabited, then $\alpha<\sigma$.
		\item If SPACE$^{\text{ITTM}}_{\alpha}$ is inhabited, then $\alpha<\sigma$.
	\end{enumerate}
\end{thm}
\begin{proof}
The argument for (1) is due to Philipp Schlicht (personal communication), the argument for (2) is completely analogous. We only give the argument for (2) and leave the adaptation to (1) to the reader.

Suppose that $X\in$SPACE$^{\text{ITTM}}_{\alpha}$ and let $P$ be an ITTM-program that decides $X$ and produces only snapshots in $L_{\alpha}[x]$ for each input $x$. Thus, the statement $\phi$ that expresses ``There is a countable ordinal $\alpha$ such that, for every $x$, $P^{x}$ halts and produces only snapshots in $L_{\alpha}[x]$'' is true in $V$. Since computations are unique, ``$P^{x}$ halts and only uses snapshots in $L_{\alpha}[x]$'' is $\Delta_{1}$. Since countable ordinals and halting ITTM-computations (which have countable length) can be encoded by real numbers, $\phi$ can be expressed as a $\Sigma^{1}_{2}$-statement.

By Shoenfield absoluteness, $\phi$ holds in $L$. By standard descriptive set theory (see e.g. \cite{MW}), the $\Sigma^{1}_{2}$-statement $\phi$ can be expressed as a $\Sigma_{1}$-statement. By $\Sigma_{1}$-elementarity, $\phi$ thus holds in $L_{\sigma}$. Consequently, there is $\alpha^{\prime}\in L_{\sigma}$ such that $L_{\sigma}$ believes that $P^{x}$ halts  and produces only snapshots in $L_{\alpha^{\prime}}[x]$ for all $x$. By absoluteness of computations, this implies that, for all $x\in L_{\sigma}$, $P^{x}$ halts and produces only snapshots in $L_{\alpha^{\prime}}[x]$. Clearly, as $\alpha^{\prime}\in L_{\sigma}$, we have $\alpha^{\prime}<\sigma$.

It thus suffices to show that $\alpha^{\prime}\geq\alpha$. If not, there is some real number $y$ such that $P^{y}$ halts and produces snapshots outside of $L_{\alpha^{\prime}}[y]$. Therefore, the statement that there is such a real number, which is $\Sigma_{1}$ in the countable parameter $\alpha^{\prime}$, holds in $V$, and thus in $L$, and thus in $L_{\sigma}$. Hence, there is a real number $y\in L_{\sigma}$ such that $P^{y}$ halts and produces a snapshot outside of $L_{\alpha^{\prime}}[y]$, a contradiction.
	
	
	
\end{proof}

\begin{remark}
	In Winter \cite{Wi1}, the notation P$^{\text{HK}}_{\alpha}$ used for 
	TIME$^{\text{ITTM}}_{\alpha}$ and the notation PSPACE$^{\text{HK}}_{\alpha}$ for SPACE$^{\text{ITTM}}_{\alpha}$;\footnote{The ``HK'' stands for Hamkins-Kidder, reminding that it was J.Hamkins and J. Kidder who originally invented ITTMs, which are therefore also called Hamkins-Kidder-machines.} with this notation, it was asked in \cite{BIWOC} (p.  42, question $1$) for which ordinals we have P$^{\text{HK}}_{\alpha}\subsetneq$PSPACE$^{\text{HK}}_{\alpha}$. In \cite{Wi1}, it is shown that P$^{\text{HK}}_{\alpha}\subseteq$PSPACE$^{\text{HK}}_{\alpha}$ for all $\alpha>\omega$ (Proposition 5.7).  The question is thus when we have inequality.  In \cite{Wi1}, this is shown for $\alpha\in(\omega,\omega_{1}^{\text{CK}}]$ (Theorem 5.14) and certain successor ordinals (Theorem 5.16). Clearly, we have P$^{\text{HK}}_{\omega_{1}}=$PSPACE$^{\text{HK}}_{\omega_{1}}$, as both coincide with the set of ITTM-decidable sets. For countable $\alpha$, Theorem \ref{sigma bound on inhabited classes} implies that both classes only make sense when $\alpha<\sigma$.  For these values of $\alpha>\omega$, Corollary \ref{ITTM time space complexity} tells us that P$^{\text{HK}}_{\alpha}\nsubseteq$PSPACE$^{\text{HK}}_{\omega+1}\subseteq$PSPACE$^{\text{HK}}_{\alpha}$ and thus in particular that P$^{\text{HK}}_{\alpha}\neq$PSPACE$^{\text{HK}}_{\alpha}$. Thus, the proper inclusion relation in question holds for all relevant values of $\alpha$.
\end{remark}

\subsection{Space and Time Complexity with Dependency on the Input}

In \cite{Wi1}, p. 78, it was asked whether P$\subsetneq$PSPACE, P$_{+}\subsetneq$PSPACE$_{+}$ and P$_{++}\subsetneq$PSPACE$_{++}$ hold for ITTMs. (The weak inclusions are both shown in \cite{Wi1}.)
We will now answer the first two questions in the positive by showing that WO belongs to SPACE$^{\text{ITTM}}_{\omega+1}$ (and thus both to PSPACE and PSPACE$_{+}$), but not to $P_{+}$.

\begin{lemma}{\label{WO counterexample}}
WO is decidable with recursive snapshots and thus belongs in particular to SPACE$^{\text{ITTM}}_{\omega+1}$, but WO does not belong to P$_{+}$. In particular, we have $P_{+}\nsupseteq$PSPACE.
\end{lemma}
\begin{proof}
That WO is decidable with recursive snapshots follows from the fact that ITRMs can decide WO (see \cite{ITRM2}, Theorem \ref{ITRM recog ITTM simple}). 

To see that WO does not belong to P$_{+}$, recall the well-known fact that there are real numbers $x$ such that $x\in L_{\omega_{1}^{\text{CK},x}}$ and $L_{\omega_{1}^{\text{CK},x}}$ believes that $x$ codes a well-founded ordering, but this is in fact false, see e.g. \cite{Go}. 

Now suppose that WO is decidable by an ITTM-program $P$ that uses $<\omega_{1}^{\text{CK},x}$ many steps on input $x$. By a slight modification of $P$, we obtain a program $Q$ that works with the same time bound and outputs $1$ on input $x$ if and only if $x$ codes a well-ordering and otherwise outputs an ill-founded sequence for $x$. To see this, let $x\subseteq\omega$ be given. We run $P$ on $x$. If the output is $1$, we halt. Otherwise, we use $P$ to test for each $i\in\omega$ whether the ordering coded by $x$ below $i$ is a well-ordering, continue up to the first $i_{0}$ for which the answer is "no" (which must necessarily exist by assumption) and write it to the output tape. We then similarly look for the first $i_{1}$ that is smaller than $i_{0}$ in the sense of the ordering coded by $x$ such that the ordering below $i_{1}$ is ill-founded and write it to the output tape. Clearly, this will generate an ill-founded sequence for $x$, and it is easy to see that the time bound $\omega_{1}^{\text{CK},x}$ is still obeyed (the function mapping $i$ to the time it takes $P$ to check whether the ordering below $i$ is well-founded is $\Sigma_{1}$ over $L_{\omega_{1}^{\text{CK},x}}$ and hence bounded).

Thus, if there was such a program $P$, then $L_{\omega_{1}^{\text{CK},x}}$ would have to contain an ill-founded sequence for $x$ whenever $x$ codes an ill-founded ordering. But this contradicts the statement just made.

\end{proof}

We now turn to the question whether P$_{++}\subsetneq$PSPACE$_{++}$, which will be treated by an application of the idea of time-bounded halting problems used in Winter \cite{Wi1} to show a few other such strict inclusions.\footnote{However, on p. 78 of \cite{Wi1}, the author conjectures that such techniques are not helpful in resolving P$_{++}\subsetneq$PSPACE$_{++}$. It turns out that, in combination with the ITRM-ITTM-simulation idea, they are. We regard this as a good example how the investigations of different models of infinitary computability can fruitfully interact. (Note that ITRMs were only introduced over a year after \cite{Wi1} was defended.)}

Below, if $x\subseteq\omega$, $n(x)$ will denote the largest natural number $k$ satisfying $k\subseteq\omega$ if there is one, and $0$ if there is none.

\begin{thm}{\label{p++ vs pspace++}}
The set $X:=\{x\subseteq\omega:P_{n(x)}^{x}\text{ does not terminate in }<\omega_{2}^{\text{CK},x}\text{ many steps}\}$ (where $P_{n}$ denotes the $n$th ITTM-program as usual) is ITRM-decidable and thus ITTM-decidable with recursive snapshots, but does not belong to P$_{++}$.
\end{thm}
\begin{proof}
To decide whether $x\in X$ for some given $x\subseteq\omega$ on an ITRM, compute a code for $L_{\omega_{3}^{\text{CK},x}}$ and use it to evaluate the $\Sigma_{1}$-statement that $P_{n(x)}^{x}$ terminates in $<\omega_{2}^{\text{CK},x}$ many steps.

To see that $X$ does not belong to P$_{++}$, suppose for a contradiction that $P$ is an ITTM-program that decides $X$ and runs for $<\omega_{2}^{\text{CK},x}$ many steps on input $x\subseteq\omega$. We modify $P$ a bit to an ITTM-program $Q$ that works as follows: On input $x$, it uses $P$ to decide whether or not $x\in X$. If $x\in X$, i.e. if $P_{n(x)}^{x}$ halts in $<\omega_{2}^{\text{CK},x}$ many steps, then $Q^{x}$ enters an infinite loop; otherwise, $Q^{x}$ halts. As $P^{x}$ runs for $<\omega_{2}^{\text{CK},x}$ many steps by assumption, $Q^{x}$ will halt in $<\omega_{2}^{\text{CK},x}$ many steps in the latter case. Let $n\in\omega$ be such that $Q=P_{n}$ and pick $x\subseteq\omega$ such that $n(x)=n$. 
Then $Q^{x}$ halts in $<\omega_{2}^{\text{CK},x}$ many steps if and only if $P_{n(x)}^{x}$ does not halt in $<\omega_{2}^{\text{CK},x}$ many steps, which, as $P_{n(x)}=P_{n}=Q$, is true if and only if $Q^{x}$ does not halt in $<\omega_{2}^{\text{CK},x}$ many steps, a contradiction.
\end{proof}

\bigskip
\noindent
\textbf{Remark}: The same technique works for any ITRM-computable function $f$ instead of $x\mapsto\omega_{1}^{\text{CK},x}$, such as $x\mapsto\omega_{n}^{\text{CK},x}$ for any $n\in\omega$, and in particular, it works for $f(x)=\omega^{\omega}$ and $f(x)=\omega_{1}^{\text{CK},x}$. Thus, P$\subsetneq$PSPACE and P$_{+}\subsetneq$PSPACE$_{+}$ can both be shown by similar arguments. The point of treating Lemma \ref{WO counterexample} separately was to give WO as a particularly natural example. 

We also remark that it was already observed by Winter that the complement of WO has a low nondeterministic ITTM-space complexity, see \cite{Wi1}, Proposition 7.18.


\section{Relations between space complexity classes for ITTMs}

We have seen above that there are sets with arbitrarily large uniform ITTM-decision times that are decidable with recursive snapshots. We know that WO has no countable bound on the decision times, but is ITTM-decidable with recursive snapshots. 

One may thus wonder: Do recursive snapshots restrict ITTMs at all? We show that this is indeed the case, and moreover, we will show that for any $\alpha<\lambda$, there are ITTM-decidable sets that are not contained in SPACE$^{\text{ITTM}}_{\alpha}$. To this end, we will show that, for any $\alpha<\lambda$, one can decide uniformly in $n$ and $x$ whether a given ITTM-program $P_{n}$ will produce a snapshot outside of $L_{\alpha}$ in the input $x$ and also whether $P_{n}^{x}$ will halt. Before we proceed, we will briefly explain the guiding idea.

As we recalled above, any ITTM-program on input $x$ will either halt in $<\lambda^{x}$ many steps or run into a strong loop by time $\Sigma^{x}$. It thus seems that there is an easy way to solve the ITTM-halting problem on an ITTM: Given $P_{n}$ and $x$, just run $P_{n}^{x}$ and keep track of all occuring snapshots in the following way: Whenever a snapshot $s$ occurs, write it to some portion of the scratch tape. If some later snapshot falls below $s$ in any component, delete $s$. Thus, if a snapshots appears that is already stored (i.e. it has appeared, but was not deleted), we know that $P_{n}^{x}$ is strongly looping and will thus halt. If this does not happen, then $P_{n}^{x}$ does halt, which will eventually also be noticed.

Clearly, there must be something wrong with this argument. It is not hard to see what: In general, there is no way to store all occuring snapshots on `some portion of the scratch tape': We do not know beforehand how many different snapshots (in the order-type of their appearance in the computation) there will be and thus do not know in how many portions to split the scratch tape. Moreover, there may be too many to do this effectively. 

On the other hand, if it is somehow ensured that all occuring snapshots can be stored, then the above works fine. Indeed, this is the idea both behind the solution of the ITRM-halting problem on ITTMs  due to Koepke and Miller \cite{ITRM2} and the ITTM-halting problem on OTMs (see L\"owe, \cite{Loe}), which have tapes of proper class length $\text{On}$ and can thus store any amount of ITTM-snapshots. 

Now, if we know that all snapshots of a computation $P_{n}^{x}$ are recursive, we can simply split the scratch tape into $\omega$ many portions and then, for each arising snapshot $s$, look for the minimal index $i$ of a (classical) Turing program $T_{i}$ that computes $s$. Then, we can store $s$ in the $i$th component of the scratch tape. 

In fact, we can do the above for any ITTM-program $P_{n}$ and any input $x$ as long as only recursive snapshots occur. Once this is violated, we can easily notice this and halt with an output indicating this.

Moreover, this approach works for any set $S$ of snapshots as long as there is an ITTM-computable bijection $f:\omega\rightarrow S$. (We may even let $S$ depend on the input $x$ and write $S_{x}$, as long as the required bijections are ITTM-computable uniformly in $x$.) As this is the case for $L_{\alpha}$ for all $\alpha<\lambda$, we obtain the following:

\begin{defini}
	For $x\subseteq\omega$, let $\mathcal{R}_{x}$ denote the set of ITTM-programs that use only recursive snapshots in the oracle $x$. Let $\bf{\mathcal{R}}$ denote the set of ITTM-programs (equivalently, their indices) that use only recursive snapshots in any real oracle.
\end{defini}

\begin{lemma}{\label{rec snapshots hp}}
	There is an ITTM-program $H$ such that the following holds: Given $n\in\mathcal{R}_{x}$, we have $H^{x}(n)\downarrow=1$ if and only if $P_{n}^{x}\downarrow$ and otherwise, $H^{x}(n)\downarrow=0$. Intuitively, $H$ solves the ITTM-halting problem for programs with recursive snapshots, uniformly in the oracle.
\end{lemma}
\begin{proof}
	Recall from above that any ITTM-program in any oracle either halts or enters a strong loop.
	
	Now, given $n$ and $x$ as in the statement of the lemma, $H$ proceeds as follows: Let $P_{n}^{x}$ run. For any snapshot that occurs, find the smallest $i$ such that the $i$th Turing program computes the snapshot (this exists by assumption). Now keep track of the occuring snapshots as follows: For each $i\in\omega$, mark the $i$th scratch tape cell with $1$ to indicate that the real number $r$ computed by the $i$th Turing program has occured as a snapshot during the computation of $P_{n}^{x}$ and that after that, no snapshot has arisen so far that is smaller than $r$ in any component. If a snapshot appears that is smaller than $r$ in any component, reset the content of the $i$th scratch tape cell to $0$.
	
	If $i$ occurs as the index of a program generating the current snapshot while there is a $1$ on the $i$th cell, we know that $P_{n}^{x}$ has entered a strong loop and $P_{n}^{x}$ will not halt; in this case, we halt with output $0$. Conversely, if $P_{n}^{x}$ does not halt, such a loop exists and will eventually be found. Thus, the program $H$ halts with output $0$ on the input $(n,x)$ when $P_{n}^{x}\uparrow$.
	
	On the other hand, when $P_{n}^{x}$ halts, then let $H(n,x)$ halt with output $1$. Clearly, $H$ is as desired.
\end{proof}

\begin{lemma}{\label{rec snapshots decidable}}
	$\mathcal{R}_{x}$ is ITTM-writable, uniformly in $x$.
\end{lemma}
\begin{proof}
 To decide whether $n\in\omega$ belongs to $\mathcal{R}_{x}$, let $H(n,x)$ run and in parallel, run $P_{n}^{x}$ and check for each snapshot that occurs during the computation whether or not it is recursive (this is easily possible on an ITTM). When a non-recursive snapshot is detected, halt with output $0$. Otherwise, $P_{n}^{x}$ uses only recursive snapshots and thus, the $H$ will successfully detect either a strong loop or the halting of $P_{n}^{x}$. In both cases, we can be sure that only recursive snapshots will occur and return $1$.
\end{proof}

\begin{corollary}{\label{semidecidable with rec snapshots}}
	If $X\subseteq\mathfrak{P}(\omega)$ is ITTM-semidecidable with recursive snapshots, then $X$ is ITTM-decidable.
\end{corollary}
\begin{proof}
	Let $P$ be an ITTM-program with recursive snapshots that semidecides $X$. Then $H$ can be used to determine whether $P$ halts on a given input $x\subseteq\omega$, and thus whether $x\in X$.
\end{proof}

\begin{thm}{\label{ITTM decidable non recursive snapshots}}
	There is a set $X\subseteq\mathfrak{P}(\omega)$ such that $X$ is ITTM-decidable, but not with recursive snapshots.
\end{thm}
\begin{proof}
	For $x\subseteq\omega$, let $n(x)$ be maximal such that $n\subseteq x$, if this $n$ is a natural number, and let $n(x)=0$ when $x=\omega$.
	
	Now let $X$ be the set of $x\subseteq\omega$ such that $n(x)\in\mathcal{R}_{x}$ (i.e. $P_{n(x)}^{x}$ only generates recursive snapshots) and $P_{n(x)}^{x}$ does not halt with output $1$ (i.e. it either halts with an output different from $1$ or it does not halt at all).
	
	Clearly, $X$ is ITTM-decidable: First, we can decide $\mathcal{R}_{x}$ by Lemma \ref{rec snapshots decidable}. Then, we can use $H$ to determine whether $P_{n(x)}^{x}$ will halt. If not, output $1$. Otherwise, run $P_{n(x)}^{x}$ and output $0$ when the output is $1$ and otherwise output $1$.
	
	Now suppose that $n\in\bf{\mathcal{R}}$ is such that $P_{n}$ decides $X$. Pick $x\subseteq\omega$ such that $n(x)=n$. Then $x\in X\leftrightarrow P_{n}^{x}\downarrow=1\leftrightarrow P_{n(x)}^{x}\downarrow=1\leftrightarrow x\notin X$, where the last implication is an equivalence because $P_{n}^{x}$ uses only recursive snapshots by assumption. This contradiction shows that $P_{n}$ cannot exist.
\end{proof}

In the investigations of space complexity for ITTMs, there seems to have been no result so far showing that there are ITTM-decidable problems that are not in SPACE$^{\text{ITTM}}_{\alpha}$ for any $\alpha>\omega$. We note that the above proof can be generalized to yield some information about this.

\begin{thm}{\label{ITTM SPACE ALPHA}}
Let $\alpha<\lambda$. Let $X_{\alpha}$ be the set of all real numbers $x$ such that all snapshots of $P_{n(x)}^{x}$ belong to $L_{\alpha}[x]$ and $P_{n(x)}^{x}$ does not halt with output $1$. Then $X_{\alpha}$ is ITTM-decidable, but does not belong to SPACE$^{\text{ITTM}}_{\alpha}$.
\end{thm}
\begin{proof}
Since $\alpha<\lambda$, there is an ITTM-writable code for $L_{\alpha}$. Thus, there is an ITTM-program $P_{\alpha-\text{test}}$ that decides $\mathfrak{P}(\omega)\cap L_{\alpha}[x]$, uniformly in $x$. Moreover, there is a uniformly in $x$ ITTM-computable bijection $f:\omega\rightarrow\mathfrak{P}(\omega)\cap L_{\alpha}[x]$. We thus obtain the obvious analogues of Lemma \ref{rec snapshots hp} and 
Lemma \ref{rec snapshots decidable}, and thus of Theorem \ref{ITTM decidable non recursive snapshots}.
\end{proof}

 To the best of our knowledge, no nontrivial proper inclusion relations between the classes SPACE$^{\text{ITTM}}_{\alpha}$ are known so far for different values of $\alpha$ bigger than $\omega$. 
(Clearly, we have SPACE$^{\text{ITTM}}_{\alpha}\subseteq$SPACE$^{\text{ITTM}}_{\beta}$ for $\alpha<\beta$, see e.g. \cite{Wi1}). 
We will now investigate the above construction a bit further, which will allow us to show that SPACE$^{\text{ITTM}}_{\alpha}\subsetneq$SPACE$^{\text{ITTM}}_{\lambda}$ for all $\alpha<\lambda$.

\begin{lemma}{\label{bounded bound violation time}}
Let $P$ be an ITTM-program, $x\subseteq\omega$, $\alpha<\lambda^{x}$, and suppose that $P^{x}$ produces a snapshot that is not contained in $L_{\alpha}[x]$. Then, letting $\tau$ be the first time at which such a snapshot appears in the computation of $P^{x}$, $\tau+\omega$ is clockable in $x$ and consequently, we have $\tau<\lambda^{x}$. 
\end{lemma}
\begin{proof}
First, compute a code for $L_{\alpha}[x]$. We can then use this to test for a given snapshot whether it is contained in $L_{\alpha}[x]$. 
Now let $P^{x}$ run and test for each snapshot whether it is contained in $L_{\alpha}[x]$. Once this fails, halt.
\end{proof}

\begin{lemma}{\label{bounded looping time}}
Let $P$ be an ITTM-program, $x\subseteq\omega$, $\alpha<\lambda^{x}$ and suppose that $P^{x}$ only uses snapshots in $L_{\alpha}[x]$. Then $P^{x}$ either halts in $<\lambda^{x}$ many steps or runs into a strong loop in $<\lambda^{x}$ many steps.
\end{lemma}
\begin{proof}
We start by computing a code $c$ for $L_{\alpha}[x]$. Each occuring snapshot is coded in $c$ by some natural number. Now, as in the solution of the halting problem for ITTMs with recursive snapshots, we can run $P^{x}$ and store the snapshots that have occured so far by the natural numbers coding them in $c$ and thus detect strong loops in the same way. As soon as a strong loop is detected, halt. Thus, there is a clockable ordinal after the first repitition in the strong loop of $P^{x}$. On the other hand, if $P^{x}$ halts, it does so in $<\lambda^{x}$ many steps. 
\end{proof}

\begin{thm}{\label{computable snapshots}}
Let $\omega<\alpha<\lambda$. Then there is a set $X\subseteq\omega$ such that $X$ is ITTM-decidable with snapshots in $L_{\lambda}$, but not with snapshots in $L_{\alpha}$.
Thus, we have SPACE$^{\text{ITTM}}_{\alpha}\subsetneq$SPACE$^{\text{ITTM}}_{\lambda}$.
\end{thm}
\begin{proof}
Let $X:=\{n\in\omega:P_{n}^{n}\text{ generates a snapshot outside of }L_{\alpha}\vee P_{n}^{n}\text{ only produces snapshots in }L_{\alpha}\text{ and runs into a strong loop}\vee P_{n}^{n}\downarrow\neq 1\text{ and uses only snapshots in }L_{\alpha}\}$. (Note that each natural number is a set of natural numbers and thus also a real number.)

Suppose that $X$ was decidable by the program $P_{n}$ using only snapshots in $L_{\alpha}$. Then $n\in X\leftrightarrow P_{n}^{n}\downarrow=1\leftrightarrow n\notin X$, a contradiction. Thus, $X$ is not decidable with snapshots in $L_{\alpha}$.

Clearly, $X$ is ITTM-decidable. 
We show that this decision procedure only uses snapshots in $L_{\lambda}$. First, testing whether $x\subseteq\omega$ is an element of $\omega$ can be done using a finite amount of memory, and thus with recursive snapshots, and in particular with snapshots in $L_{\omega+1}\subseteq L_{\alpha}$. 
Now, given that $x\in\omega$ and thus that $\lambda^{x}=\lambda$, we know from Lemma \ref{bounded bound violation time} and Lemma \ref{bounded looping time} that $P_{n}^{x}$ will produce a snapshot outside of $L_{\alpha}$, halt or loop before $\lambda$. In the latter two cases, all occuring snapshots will thus be contained in $L_{\lambda}$. 
Moreover, each initial segment (before the second loop, if the computation is looping) of the computation of $P_{n}^{x}$ will be an element of $L_{\lambda}$ and hence so will be the lists of natural numbers representing the snapshots that are used in the decision procedure. 
In the first case, we know from Lemma \ref{bounded bound violation time} that the first snapshot outside of $L_{\alpha}$ will occur in $<\lambda$ many steps and thus also be contained in $L_{\alpha}$.

Thus, the whole decision procedure indeed only uses snapshots in $L_{\lambda}$, hence $X\in$SPACE$^{\text{ITTM}}_{\lambda}$. As $X\notin$SPACE$^{\text{ITTM}}_{\alpha}$ by the argument above, we have SPACE$^{\text{ITTM}}_{\lambda}\setminus$SPACE$^{\text{ITTM}}_{\alpha}\neq\emptyset$, as desired.

\end{proof}

\subsection{Separating Space Complexity Classes for ITTMs}

The next natural questions are now whether there are sets that are ITTM-decidable, but not with snapshots in $L_{\lambda}$ (i.e. whether SPACE$^{\text{ITTM}}_{\lambda}$ equals the class of ITTM-decidable sets) and whether there are proper inclusions between the classes SPACE$^{\text{ITTM}}_{\alpha}$ for different infinite values of $\alpha<\lambda$.

The former question will be answered below as a consequence of the more general Corollary \ref{increasing at lambda}. 
For the time being, we offer a partial result. To this end, we introduce notions that we believe to be interesting in their own right.

\begin{defini}
Let $X,C$ be sets of real numbers. We say that $X$ is ITTM-semidecidable with snapshots in $C$ if and only if there is an ITTM-program $P$ that semidecides $X$ and produces only snapshots in $C$ on any input $x$. 
If $C=L_{\alpha}$ for some $\alpha\in\text{On}$, we write sSPACE$^{\text{ITTM}}_{\alpha}$ for the set of sets that are ITTM-semidecidable with snapshots in $C$. 
\end{defini}

\begin{prop}
There is an ITTM-semidecidable set $X$ that is not ITTM-semidecidable with snapshots in $L_{\lambda}$. Thus, sSPACE$^{\text{ITTM}}_{\lambda}\subsetneq$sSPACE$^{\text{ITTM}}_{\omega_{1}}$.
\end{prop}
\begin{proof}
Let $X$ be the set of $x\subseteq\omega$ such that $P_{n(x)}^{x}$ runs into a strong loop and generates only snapshots in $L_{\lambda}$.

Suppose for a contradiction that $P_{n}$ is an ITTM-program that semidecides $X$ and uses only snapshots in $L_{\lambda}$ on any input. Let $x\subseteq\omega$ be such that $n(x)=n$. 
Then $P_{n(x)}^{x}$ does by definition not generate snapshots outside of $L_{\lambda}$. Moreover, we have 
$x\notin X\leftrightarrow P_{n(x)}^{x}\downarrow\leftrightarrow P_{n}^{x}\downarrow\leftrightarrow x\in X$, a contradiction. Hence $X$ is not ITTM-decidable with snapshots in $L_{\lambda}$.

On the other hand, $X$ is ITTM-semidecidable: Given $x\subseteq\omega$, we let $P_{n(x)}^{x}$ run. For every snapshot $s$ generated, we run all ITTM-programs simultaneously, waiting for one to halt with output $s$. If this happens, the index of that program is used to store $s$ as in the solution to the halting problem for ITTMs with recursive snapshots above. Since we use the index of the machine that halts first (the minimal one if there is more than one) with output $s$, the same snapshot will always receive the same index. Thus, if $P_{n(x)}^{x}$ generates only recursive snapshots, we will either eventually observe that it halts or find a witness for a strong looping, as in the solution to the halting problem above. In the former case, we enter an endless loop, in the latter case, we halt. On the other hand, if $P_{n(x)}^{x}$ does generate snapshots outside of $L_{\lambda}$, the search for an ITTM-program that generates $s$ will not terminate, and hence our procedure does not halt. Thus, our procedure halts on input $x$ if and only if $x\in X$, as desired.
\end{proof}


We now consider the problem of separating the classes SPACE$^{\text{ITTM}}_{\alpha}$ for different values of $\alpha$. It is natural to attempt adapting the separation of SPACE$^{\text{ITTM}}_{\lambda}$ from the set of ITTM-decidable sets to this purpose. We first consider a seemingly unrelated question; the desired separation will result as a by-product.


The proof of the unrelatedness of space and time complexity for ITTMs above relied on decidable singletons. Is there also a singleton set $\{x\}$ with the property that $\{x\}$ is ITTM-decidable, but not, say, with recursive snapshots? Indeed, there is. In fact, much more holds, as we will now show.

\begin{thm}{\label{complex singletons}}
	Let $\alpha<\lambda$. Then there is a real number $x$ such that $x$ is ITTM-recognizable, but not with snapshots in $L_{\alpha}$. In fact, $x$ can be taken to be ITTM-writable.
\end{thm}
\begin{proof}
	Let $x\in (L_{\zeta}\setminus L_{\lambda})\cap\mathfrak{P}(\omega)$ be such that $x\in L_{\lambda^{x}}$. (One can e.g. take $x$ to be an eventually writable code for some ordinal in the interval $[\lambda,\zeta)$.) 
	Then $x$ is not ITTM-recognizable 
	and a fortiori not ITTM-recognizable with snapshots in $L_{\alpha}$. 
	
	As the eventually writable real numbers are closed under ITTM-writability, we have $\lambda^{x}\leq\zeta$. 
	Now, for each $n\in\omega$, by Lemma \ref{bounded looping time} and Lemma \ref{bounded bound violation time}, $L_{\lambda^{x}}$ contains one of the following:
	
	\begin{enumerate}
		\item A partial computation of $P_{n}^{x}$ that contains a snapshot outside of $L_{\alpha}[x]$.
		\item A strong loop of $P_{n}^{x}$.
		\item A terminating computation of $P_{n}^{x}$ with output $\neq 1$.
		\item A terminating computation of $P_{n}^{x}$ with output $1$.
	\end{enumerate}
	
	In cases (1)-(3), we thus know that $P_{n}$ does not recognize $x$ with snapshots in $L_{\alpha}$. In case $4$, we know that the $\Sigma_{1}$-formula $\exists{x}P_{n}^{x}\downarrow=1$ holds in $L_{\zeta}$, and thus, as $L_{\lambda}\prec_{\Sigma_{1}}L_{\zeta}$, it also holds in $L_{\lambda}$. Thus, $L_{\lambda}$ contains some $y$ such that $P_{n}^{y}\downarrow=1$. As $x\notin L_{\lambda}$ by assumption, we have $y\neq x$. Thus, $P_{n}$ outputs $1$ on two different inputs and therefore also does not recognize $x$. Moreover, as $\lambda\leq\lambda^{x}$, we have $x,y\in L_{\lambda^{x}}$, so in case (4), $L_{\lambda^{x}}$ contains two different real numbers for which $P_{n}$ outputs $1$.
	 In other words, $L_{\lambda^{x}}$ witnesses that $x$ is not ITTM-recognizable with snapshots in $L_{\alpha}$. 
	 
	 Thus, $L_{\Sigma}$ satisfies the $\Sigma_{1}$-formula that there are $x\subseteq\omega$ and an $L$-level $L_{\beta}$ such that, for every $n\in\omega$, $L_{\beta}$ contains one of (1)-(4), thus witnessing that $x$ is not ITTM-recognizable with snapshots in $L_{\alpha}$.
	
	As $L_{\lambda}\prec_{\Sigma_{1}}L_{\Sigma}$, the same holds in $L_{\lambda}$. Thus, there is a real number $r$ in $L_{\lambda}$ that is not ITTM-recognizable with snapshots in $L_{\alpha}$. Since $r\in L_{\lambda}$, $r$ is ITTM-writable and thus a fortiori ITTM-recognizable. Thus, $r$ is as desired.
\end{proof}

As a corollary, we obtain that many of the complexity classes SPACE$^{\text{ITTM}}_{\alpha}$ with $\alpha<\lambda$ are distinct:

\begin{corollary}{\label{proper inclusions below lambda}}
	For all $\alpha<\lambda$, there is $\beta\in(\alpha,\lambda)$ such that SPACE$^{\text{ITTM}}_{\alpha}\subsetneq$SPACE$^{\text{ITTM}}_{\beta}$.
\end{corollary}
\begin{proof}
	Given $\alpha<\lambda$, the proof of Theorem \ref{complex singletons} shows how to obtain an ITTM-writable real number $x$ that is not ITTM-recognizable with snapshots in $L_{\alpha}$, i.e. such that $\{x\}\notin$SPACE$^{\text{ITTM}}_{\alpha}$. 
	
	Let $\beta$ be the writing time of $x$. Then $\beta<\lambda$. In order to determine whether some $y\subseteq\omega$ given in the oracle is equal to $x$, it suffices to write $x$ and then to compare it to $y$ bit by bit. The bit-by-bit-comparison can be done with finitely many memory bits besides $x$ and $y$. Writing $x$, on the other hand, only requires snapshots in $L_{\beta}$. Thus, $x$ is ITTM-recognizable with snapshots in $L_{\beta}$, i.e. 
	$\{x\}\in$SPACE$^{\text{ITTM}}_{\alpha}\subsetneq$SPACE$^{\text{ITTM}}_{\beta}$.
	
	Thus, we have SPACE$^{\text{ITTM}}_{\alpha}\subsetneq$SPACE$^{\text{ITTM}}_{\beta}$.
\end{proof}

Relativizing the above proof, we obtain a more general statement:

\begin{lemma}{\label{unbounded in relativized lambdas}}
	Let $c$ be ITRM-recognizable.
	Then there are cofinally in $\lambda^{c}$ many ordinals $\alpha$ such that SPACE$^{\text{ITTM}}_{\alpha}$ is inhabited.
\end{lemma}
\begin{proof}
	Let $c$ be ITRM-recognizable. Then $c$ is in particular ITTM-recognizable and thus we have $c\in L_{\lambda^{c}}$. 
	
	We claim that no element of $L_{\Sigma^{c}}\setminus L_{\lambda^{c}}$ is ITTM-recognizable. Otherwise, if $s\in L_{\Sigma^{c}}\setminus L_{\lambda^{c}}$ was recognizable by the ITTM-program $Q$, we could let the universal ITTM-program $\mathcal{U}$ run in the oracle $c$, test each snapshot with $Q$ and halt once $s$ shows up. This program would halt after $\geq\lambda^{c}$ many steps in the oracle $c$, a contradiction.
	
	Let $d$ be the $<_{L}$-minimal code for $\lambda^{c}$. Then $d\in L_{\zeta^{c}}$, and thus $c\oplus d\in L_{\zeta^{c}}$. Moreover, $d$ is not writable in $c$, so we have $c\oplus d\in L_{\zeta^{c}}\setminus L_{\lambda^{c}}$. Consequently, $c\oplus d$ is not ITTM-recognizable. Moreover, we have $\lambda^{c\oplus d}\leq\zeta^{c}<\Sigma^{c}$.
	
	Now fix $\alpha\in(\omega,\lambda^{c})$.
	
	As in the proof of Theorem \ref{complex singletons}, $L_{\lambda^{c\oplus d}}$ will witness that $c\oplus d$ is not recognizable by an ITTM-program that uses only snapshots in $L_{\alpha}$. Thus, the statement that there are a real number $d$ and an ordinal $\delta$ such that $L_{\delta}$ witnesses that $c\oplus d$ is not ITTM-recognizable with snapshots in $L_{\alpha}$ holds in $L_{\Sigma^{c}}$. Clearly, this statement is $\Sigma_{1}$ in the parameters $\alpha$ and $c$, which are both contained in $L_{\lambda^{c}}$. 
	
	As $L_{\lambda^{c}}\prec_{\Sigma_{1}}L_{\Sigma^{c}}$, the same statement holds in $L_{\lambda^{c}}$. Thus, there are $d^{\prime}, L_{\delta^{\prime}}\in L_{\lambda^{c}}$ such that $L_{\delta^{\prime\prime}}$ witnesses that $c\oplus d^{\prime}$ is not ITTM-recognizable with snapshots in $L_{\alpha}$. Consequently, we have $\{c\oplus d^{\prime}\}\notin$SPACE$^{\text{ITTM}}_{\alpha}$.
	
	We now show that there is $\beta<\lambda^{c}$ such that $\{c\oplus d^{\prime}\}\in$SPACE$^{\text{ITTM}}_{\beta}$. It then follows that, for some $\eta>\alpha$, SPACE$^{\text{ITTM}}_{\eta}$ must be inhabited, as desired. As $d^{\prime}\in L_{\lambda^{c}}$, $d^{\prime}$ is writable in the oracle $c$; let $Q$ be an ITTM-program such that $Q^{c}$ write $d^{\prime}$.
	
	By assumption, $c$ is ITRM-recognizable and thus ITTM-recognizable with recursive snapshots. Now, given $z=z_{0}\oplus z_{1}\subseteq\omega$ in the oracle, we first use this fact to test, using only recursive snapshots, whether $z_{0}=c$. If not, we halt with output $0$. Otherwise, we know that $z_{0}=c$ and we run $Q^{z_{0}}$. By the choice of $Q$, the output must be $d^{\prime}$, which can now be compared to $z_{1}$ bit by bit using only a finite amount of memory. Now, if $\tau$ is the halting time of $Q^{c}$, then $\tau<\lambda^{c}$ and the procedure just described only uses snapshots in $L_{\tau+\omega}$. Thus, we have $\{c\oplus d^{\prime}\}\in$SPACE$^{\text{ITTM}}_{\tau}$.
	
	Taken together, we have $\{c\oplus d^{\prime}\}\in$SPACE$^{\text{ITTM}}_{\tau}\setminus$SPACE$^{\text{ITTM}}_{\alpha}$, as desired.
\end{proof}

\begin{thm}{\label{many points of increase}}
There are cofinally in $\sigma$ many $\alpha$ such that SPACE$^{\text{ITTM}}_{\alpha}$ is inhabited. Thus, for every $\alpha<\sigma$, there is $\beta\in(\alpha,\sigma)$ such that SPACE$^{\text{ITTM}}_{\alpha}\subsetneq$SPACE$^{\text{ITTM}}_{\beta}$.
\end{thm}
\begin{proof}
Recall from above that there are cofinally in $\sigma$ many ordinals $\alpha$ that have ITRM-recognizable codes. Now pick such an ordinal $\tau$ above $\alpha$ along with an ITRM-recognizable code $c$ for $\tau$ and use Lemma \ref{unbounded in relativized lambdas}.
\end{proof}

The proof of Lemma \ref{unbounded in relativized lambdas} further yields many specific separation results concerning space complexity classes for ITTMs:

\begin{corollary}{\label{increasing at lambda}}
	Let $c$ be ITRM-recognizable and $\alpha<\lambda^{c}$. Then 
	\begin{center}
		SPACE$^{\text{ITTM}}_{\alpha}\subsetneq$SPACE$^{\text{ITTM}}_{\lambda^{c}}\subsetneq$SPACE$^{\text{ITTM}}_{\sigma}$
	\end{center}
\end{corollary}
\begin{proof}
	Let $\alpha<\lambda^{c}$. Then pick $\eta\in(\alpha,\lambda^{c})$ as in the proof of Lemma \ref{unbounded in relativized lambdas}. Moreover, use Theorem \ref{many points of increase} to pick $\beta\in(\lambda^{c},\sigma)$ such that SPACE$^{\text{ITTM}}_{\lambda^{c}}\subsetneq$SPACE$^{\text{ITTM}}_{\beta}$. 
	Then we have 
	\begin{center}
SPACE$^{\text{ITTM}}_{\alpha}\subsetneq$SPACE$^{\text{ITTM}}_{\eta}\subseteq$SPACE$^{\text{ITTM}}_{\lambda^{c}}\subsetneq$SPACE$^{\text{ITTM}}_{\beta}\subseteq$SPACE$^{\text{ITTM}}_{\sigma}$,
	\end{center}
as desired.
\end{proof}

\begin{corollary}
		For each $\alpha<\sigma$, we have SPACE$^{\text{ITTM}}_{\alpha}\subsetneq$SPACE$^{\text{ITTM}}_{\sigma}$.
\end{corollary}
\begin{proof}
	Pick $\beta$ as in Lemma \ref{many points of increase}. Then SPACE$^{\text{ITTM}}_{\alpha}\subsetneq$SPACE$^{\text{ITTM}}_{\beta}\subseteq$SPACE$^{\text{ITTM}}_{\sigma}$.
\end{proof}

	

\section{Nondeterministic ITTM-complexity}

We conclude with some observations on the nondeterministic analogues of TIME$^{\text{ITTM}}_{\alpha}$ and SPACE$^{\text{ITTM}}_{\alpha}$. This study was started in L\"owe \cite{Loe} and continued by Winter \cite{Wi1} (see the discussion in the last section).

We begin by observing that WO has a high time complexity even if nondeterminism is allowed.

\begin{prop}
	WO$\notin$NTIME$^{\text{ITTM}}_{\alpha}$ for all countable $\alpha$.
\end{prop}
\begin{proof}
	Suppose otherwise, and let $\alpha$ be countable such that WO belongs to NTIME$^{\text{ITTM}}_{\alpha}$. Moreover, let $c$ be a real number coding $\alpha$ and let $P$ be a nondeterministic ITTM-program that decides WO. Now, for any $x\subseteq\omega$, the statement `There is a computation of $P^{x}$ of length $\alpha$ that halts with output $1$' is $\Sigma_{1}^{1}$ in the parameter $c$ and characterizes WO. However, it is well-known that WO is not $\Sigma_{1}^{1}$ in any real parameter, see, e.g., \cite{MW}.
\end{proof}

\begin{corollary}
	NTIME$^{\text{ITTM}}_{\alpha}\nsupseteq$SPACE$^{\text{ITTM}}_{\omega+1}$ for $\alpha$ countable. 
	
	In particular, NTIME$^{\text{ITTM}}_{\alpha}\neq$NSPACE$^{\text{ITTM}}_{\alpha}$ for $\alpha>\omega$ countable.
\end{corollary}
\begin{proof}
	The first claim follows from the last Proposition and the fact that WO belongs to SPACE$^{\text{ITTM}}_{\omega+1}$. The second is an easy consequence of the first.
\end{proof}

If we knew that NTIME$^{\text{ITTM}}_{\alpha}\subseteq$NSPACE$^{\text{ITTM}}_{\alpha}$ holds in general, we could conclude that a proper inclusion relation holds for all $\alpha$. However, this is only known for recursive $\alpha$ and open for all other values of $\alpha$ (see \cite{Wi1}).

In fact, it is quite conceivable that this inclusion fails for some values of $\alpha$: The ability of a nondeterministic ITTM to `guess' an arbitrary real number allows to nondeterministically decide various sets with a uniform time bound, while it is not clear at all how to do this with any space bound. The following proposition provides a wealth of potential counterexamples.

\begin{prop}
	For $x\subseteq\omega$, $\{x\}$ belongs to NTIME$^{\text{ITTM}}_{\alpha}$ for some $\alpha$ if and only if $x\in L_{\sigma}$.
	
	Moreover, $L_{\sigma}\cap\mathfrak{P}(\omega)$ belongs to NTIME$_{\sigma}$.
\end{prop}
\begin{proof}
Let $x\in L_{\sigma}$. Pick $\beta<\sigma$ such that $x\in L_{\beta}$ and such that, for some $\Sigma_{1}$-statement $\phi$, $\beta$ is minimal with $L_{\beta}\models\phi$. 
Let $c$ be the $<_{L}$-minimal real number that codes $L_{\beta}$. Then $\{c\}$ is ITTM-decidable in $<\sigma$ many steps, see \cite{CSW}. 
Thus, $\{x\}$ is nondeterministically decidable in $<\sigma$ steps as follows: Given the input $y\subseteq\omega$, first use $\omega$ many steps to guess a real number $d$. Then verify whether $d=c$. If not, reject. 
Otherwise, $x$ is coded in $c$ by some fixed natural number $j$ and it takes $<\sigma$ many steps to check whether $y=x$ relative to $c$.

To decide whether $x\in L_{\sigma}$, guess again a real number in $\omega$ many steps, then verify whether it codes some minimal $L$-level $L_{\beta}$ satisfying some $\Sigma_{1}$-statement $\phi$ and containing $x$. If not, reject, otherwise accept.
\end{proof}

\subsection{An alternative approach to nondeterminism}

Above, we considered nondeterministic ITTM-complexity classes defined via nondeterministic ITTMs. Another natural definition, used by Schindler in \cite{Schi} and also used in L\"owe \cite{Loe} and \cite{Wi1} would be that a set $X$ belongs to NTIME$_{f}^{\text{ITTM},*}$ if and only if there is a set $Y\subseteq\mathfrak{P}(\omega)\times\mathfrak{P}(\omega)$ such that $Y$ belongs to TIME$_{f}^{\text{ITTM}}$ and $X=\{x\subseteq\omega:\exists{y\subseteq\omega}(x,y)\in Y\}$ (and similarly for NSPACE$_{f}^{\text{ITTM},*}$). Let us denote by NDEC$^{\text{ITTM},*}$ the set of sets that are nondeterministically ITTM-decidable in this sense. With this definition, the argument for the inequality of NTIME$_{\alpha}$ and NSPACE$_{\alpha}$ still works; however, it is now possible to prove a rather strong inclusion result:

\begin{thm}{\label{alternative nondeterminism}}
For all $\alpha<\omega_{1}$, we have NDEC$^{\text{ITTM},*}\subseteq$NSPACE$^{\text{ITTM},*}_{\omega+1}$.
\end{thm}
\begin{proof}
Suppose that $X$ is nondeterministically decidable by the ITTM-program $P$. Let $Y=\{(x,c):c$ codes an accepting $P$-computation on input $x\}$. Then $Y$ is ITRM-decidable and thus belongs to SPACE$^{\text{ITTM}}_{\omega+1}$ and moreover, $X$ is the projection of $Y$ to its first component. Hence $X$ belongs to NSPACE$^{\text{ITTM},*}_{\omega+1}$.
\end{proof}

\section{Further Work}

The idea of the space complexity measures studied above is that the complexity of a snapshot is the index of the first $L$-level (relativized to the input) at which it appears. 
One natural alternative proposal (also to be found in \cite{Loe}) would be to take $\omega_{1}^{\text{CK},s}$ as a measure for the complexity of the snapshots $s$ (another possibility would be $\lambda^{x}$; note that the input does not count when measuring snapshot complexity). Clearly, when all snapshots are contained in a certain $L$-level, then the corresponding $\omega_{1}^{\text{CK},s}$'s will also be `small'; in particular, if a computation uses only recursive snapshots, then we will have $\omega_{1}^{\text{CK},s}=\omega_{1}^{\text{CK}}$ for any occuring snapshot. It is then not hard to see that the central results of this paper, such as Theorem \ref{space and time for ITTMs}, Corollary \ref{ITTM time space complexity}, Theorem \ref{sigma bound on inhabited classes}, Lemma \ref{WO counterexample}, Theorem \ref{p++ vs pspace++}, Theorem \ref{computable snapshots}, Theorem \ref{complex singletons}, Theorem \ref{many points of increase} and Corollary \ref{increasing at lambda}, will go through when one (re)defines the space usage of a computation as the supremum of $\omega_{1}^{\text{CK},s}$ for all occuring snapshots $s$ and SPACE$^{\text{ITTM}}_{\alpha}$ as the set of sets that are decidable by programs with space usage $<\alpha$ on any input, as it is proposed in \cite{Loe}. 
However, this is not clear for the proposal to use $\lambda^{s}$ as a measure for the complexity of $s$. (Observe that, if $x$ is e.g. Cohen-generic over $L_{\Sigma+1}$, we would have $\lambda^{x}=\lambda$ even if $x$ is very high up in the constructible hierarchy or not constructible at all.) It this thus possible that some interesting new phenomena may show up with this notion of space complexity.

\bigskip

A topic that was not considered here in detail was the relation of the nondeterministic
variants of the respective ITTM-complexity classes. Such concepts
were considered in \cite{Wi1}, where it is e.g. shown that NTIME$^{\text{ITTM}}_{\alpha}\subseteq$NSPACE$^{\text{ITTM}}_{\alpha}$
for $\alpha\in(\omega,\omega_{1}^{\text{CK}})$ (Proposition 7.21). The question whether this holds in general appears to be open. It may be worthwhile to see whether the methods developed in this paper can also shed light on these classes.

\bigskip

We conclude with some questions suggested by, but left open in this paper. 

\begin{question} Define "writable with recursive snapshots" and "clockable with recursive snapshots" in the obvious way, along with "writable/clockable with snapshots in $L_{\alpha}$, where we imagine that the output is written to an extra "write only"-tape $t$, the contents of which do not count in measuring space complexity. What is the supremum of the ordinals that are clockable/writable with recursive snapshots/snapshots in $L_{\alpha}$?
\end{question}

\begin{question}
	Although we know now that there are many proper inclusions among the classes SPACE$^{\text{ITTM}}_{\alpha}$ for different $\alpha<\lambda$, we do not know where they are. We do e.g. not know whether we can have SPACE$^{\text{ITTM}}_{\alpha}=$SPACE$^{\text{ITTM}}_{\alpha+1}$ for any $\alpha\in(\omega,\lambda)$ or whether such inclusions are always proper. In the former case, it would be interesting to see what the next strictly larger stage after a given $\alpha$ is.
	
	The same questions can of course be asked concerning SPACE$^{\text{ITTM}}_{\alpha}$ for $\alpha\in[\lambda,\sigma)$, i.e. for $\alpha$ below $\lambda^{c}$ for an ITRM-recognizable $c$.
\end{question}

\begin{question}
Which of the above results have analogues for $\alpha$-ITTMs or Ordinal Turing Machines (OTMs)?
\end{question}

Finally, the notion of semidecidable complexity, i.e. the "sSPACE"-classes introduced above, clearly yields as many questions as the original notion of decidable complexity.

\end{document}